\documentclass[11pt,a4paper]{amsart}
\usepackage{fullpage}
\usepackage{epsfig, amsmath,xspace}
\usepackage{amssymb,amscd,mathrsfs}
\usepackage[all,cmtip]{xy}
\usepackage{hyperref}

\newtheorem{thm}{Theorem}[section]
\newtheorem{lemma}[thm]{Lemma}
\newtheorem{proposition}[thm]{Proposition}
\newtheorem{corollary}[thm]{Corollary}
\theoremstyle{definition}
\newtheorem{remark}[thm]{Remark}
\newtheorem{defn}[thm]{Definition}
\newtheorem{example}[thm]{Example}

\def\ie{\emph{i.e.\,\,}}
\def\eg{\emph{e.g.\,\,}}
\newdir{ >}{{}*!/-7pt/@{>}}
\def\Hom{\text{Hom}}
\newcommand{\mcL}{\mathcal{L}}

\def\id{\mathrm{id}}

\def\leq{\leqslant}
\def\geq{\geqslant}
\def\ra{\rightarrow}

\newcommand{\Epin}{{\rm Epi}_n}
\newcommand{\Epinp}{{\rm Epi}_n^+}
\newcommand{\tree}{ \xymatrix{[r_n] \ar[r]^{f_n} & ...\ar[r]^{f_2}&[r_1]}}
\newcommand{\Tot}{{ \rm Tot}}

\newcommand{\Nat}{{\rm Nat}}

\newcommand{\treeR}{ \xymatrix{[r_n] \ar[r]^{f_n^r} & ...\ar[r]^{f_2^r}&[r_1]}}
\newcommand{\treeS}{ \xymatrix{[s_n] \ar[r]^{f_n^s} & ...\ar[r]^{f_2^s}&[s_1]}}

\newcommand{\Ext}{{ \rm Ext}}
\newcommand{\Tor}{{ \rm Tor}}
\newcommand{\kmod}{k\text{-}\mathrm{mod}}

\newcommand{\sh}{{ \rm sh}}

\newcommand{\op}{\mathrm{op}}
\newcommand{\Epi}{{\rm Epi}}

\begin{document}
\title{$E_n$-cohomology with coefficients as functor cohomology}
\author{Stephanie Ziegenhagen}
\address{Universit\'e Paris 13, Sorbonne Paris Cit\'e, LAGA, CNRS (UMR 7539)\\
 99 Avenue Jean-Baptiste Cl\'ement, 93430 Villetaneuse, France}
\email{ziegenhagen@math.univ-paris13.fr}

\begin{abstract}
Building on work of Livernet and Richter, we prove that $E_n$-homology and $E_n$-cohomology of a commutative algebra with coefficients in a symmetric bimodule can be interpreted as functor homology and cohomology. Furthermore we show that the associated Yoneda algebra is trivial.
\end{abstract}

\maketitle

\section{Introduction}

The little $n$-cubes operad was introduced to study $n$-fold loop spaces (see \cite{BV73}, \cite{Ma72}). An $E_n$-operad is a $\Sigma_*$-cofibrant operad weakly equivalent to the operad formed by the  singular chains on the little $n$-cubes operad, and algebras over such an operad are called $E_n$-algebras. Those are $A_{\infty}$-algebras which are in addition commutative up to higher homotopies of a certain level depending on $n$. 
For a $\Sigma_*$-cofibrant operad one can define a suitable notion of homology and cohomology of algebras over this operad as a derived functor. For $E_1$-algebras this operadic notion of homology coincides with Hochschild homology. For $E_{\infty}$-algebras one retrieves $\Gamma$-homology as defined by Robinson, see \cite{RW02}. In general, for a commutative algebra viewed as an $E_n$-algebra, $E_n$-homology can be seen to coincide with higher order Hochschild homology as defined in \cite{Pi00}, see \cite{GTZ} and \cite{Z}.

Many notions of homology can be expressed as functor homology. The case of Hochschild homology and cyclic homology has been studied by Birgit Richter and Teimuraz Pirashvili in \cite{PR02}. The same authors give a functor homology interpretation of $\Gamma$-homology in \cite{PR00}. In \cite{HV}, Eric Hoffbeck and Christine Vespa show that Leibniz homology of Lie algebras is functor homology. A more general approach to functor homology for algebras over an operad and their operadic homology is discussed in \cite{FrNote} by Benoit Fresse.

For the case of $E_n$-homology, functor homology interpretations of $E_n$-homology have been given by  Livernet-Richter in \cite{LR11} and Fresse in \cite{Fr}. Both articles are exclusively concerned with the case of trivial coefficients. As proved in \cite{Fr11}, $E_n$-homology with trivial coefficients  coincides up to a suspension with the homology of a generalized iterated bar construction. Muriel Livernet and Birgit Richter use this in \cite{LR11} to prove that $E_n$-homology of a commutative algebra with trivial coefficients can be interpreted as functor homology. Fresse shows in \cite{Fr} that this result can be extended to arbitrary $E_n$-algebras.

Recent work by Benoit Fresse and the author shows that $E_n$-homology and $E_n$-cohomology of a commutative algebra with coefficients in a symmetric bimodule can also be calculated via the iterated bar construction. We show in this article that the  functor homology interpretation of Livernet and Richter can be extended to the case with coefficients, and also holds for cohomology:

\begin{thm} Let $k$ be a commutative unital ring, $A$ a commutative nonunital algebra and $M$ a symmetric $A$-bimodule. There are a small category $\Epinp$ and  functors $b\colon {\Epinp}^{\op} \ra \kmod$ and $\mathcal{L}(A;M)\colon \Epinp \ra \kmod$ such that
$$H^{E_n}_*(A;M) \cong \Tor^{\Epinp}_*(b, \mathcal{L}(A;M)).$$
If $k$ is self-injective, we also have that
$$H_{E_n}^*(A;M) \cong \Ext_{\Epinp}^*(b, \mathcal{L}^{c}(A;M))$$
for a certain functor $\mcL^c(A;M) \colon {\Epinp}^{\op} \ra \kmod$.
\end{thm}

This implies that there is an action on $E_n$-cohomology by the corresponding Yoneda algebra. We show that this algebra is trivial.

\textbf{Outline:} We give an overview of the constructions of \cite{LR11} in section 2. In section 3 we recall how to calculate $E_n$-homology and -cohomology of commutative algebras with coefficients in a symmetric bimodule via the iterated bar construction. To do this one introduces a twisting differential. In section 4 we enlarge the category defined by Livernet and Richter to incorporate this twisting differential. We define $E_n$-homology and -cohomology for functors from this category to $k$-modules. Finally we show that there are Loday functors linking these notions to the usual notion of $E_n$-homology and -cohomology. We prove our main theorem in section 5. In section 6 we recall the definition of the Yoneda pairing and show that the Yoneda algebra is trivial.

\subsection*{Acknowledgements}
The contents of this article were part of the authors PhD thesis. I am  indebted to my advisor Birgit Richter for numerous ideas and discussions. I would also like to thank Eric Hoffbeck and Marc Lange for many helpful suggestions on how to improve the exposition of this article. Furthermore I would like to gratefully acknowledge support by the DFG.

\subsection*{Conventions} In the following we assume that $1\leq n< \infty$. Let $k$ be a commutative unital ring. We denote by $A$ a commutative nonunital $k$-algebra and by $M$ a symmetric $A$-bimodule. We often view $A$ and $M$ as differential graded $k$-modules concentrated in degree zero. Let $A_+=A \oplus k$ be the unital augmented algebra obtained by adjoining a unit to $A$.
We denote by $sc \in \Sigma C$ the element defined by $c \in C$ in the suspension of a graded $k$-module $C$. The $k$-module $k[X]$ is the free $k$-module generated by a set $X$. For $l \geq 0$ we denote by $[l]$ the set $[l]=\lbrace 0,...,l \rbrace$.


\section{The category $\Epin$ encoding the $n$-fold bar complex}

In \cite{Fr11} Benoit Fresse proves that $E_n$-homology of $E_n$-algebras with trivial coefficients can be computed via the iterated bar complex. Muriel Livernet and Birgit Richter use this in \cite{LR11} to give an interpretation of $E_n$-homology of commutative algebras with trivial coefficients as functor homology. 
They encode the information necessary to define an iterated bar complex in a category $\Epin$ of trees. We recall the construction of this category.

\begin{defn}
Let $C$ be a differential graded nonunital algebra. The bar complex $B(C)$ is the differential graded $k$-module given by
$$B(C) = (\overline T^c(\Sigma C), \partial_B),$$
where $\overline T^c(\Sigma C)$ denotes the reduced tensor coalgebra on $\Sigma C$ equipped with the differential induced by the differential of $C$. The twisting cochain $\partial_B$ is defined by
$$\partial_B([c_1\vert ... \vert c_l]) = \sum_{i=1}^{l-1} (-1)^{i-1}[c_1 \vert ... \vert c_i  c_{i+1} \vert ... \vert c_l].$$ 
Here we use the classical bar notation and denote $sc_1 \otimes ... \otimes sc_l \in (\Sigma  C)^{\otimes l}$ by $[c_1\vert ... \vert c_l]$.
If $C$ is commutative, the shuffle product 
$$ \sh \colon B(C) \otimes B(C) \ra B(C)$$  is defined by
$$\sh([c_1 \vert ... \vert c_j] \otimes [c_{j+1} \vert...\vert c_{j+l}]) 
=\sum_{\sigma \in \sh(j,l)} \pm [c_{\sigma^{-1}(1)} \vert...\vert c_{\sigma^{-1}(j+l)}]$$
with $\sh(j,l) \subset \Sigma_{j+l}$ the set of $(j,l)$-shuffles. For homogeneous elements $c_1,...,c_{j+l}$ the 
summand $[c_{\sigma^{-1}(1)} \vert...\vert c_{\sigma^{-1}(j+l)}]$ is decorated by the graded signature $
(-1)^{\epsilon}$ with $\epsilon= \prod_{i<l, \sigma(i)>\sigma(l)} (\vert c_i\vert +1)(\vert c_l\vert +1)$.
The shuffle product makes $B(C)$ a commutative differential graded $k$-algebra.
\end{defn}

In particular we can iterate this construction and form the $n$-fold bar complex $B^n(A)$. The results in \cite{Fr11} for $E_n$-algebras imply that for any $k$-projective commutative nonunital $k$-algebra $A$
$$H_*^{E_n}(A;k) = H_*(\Sigma^{-n} B^n(A)).$$

Elements in the $n$-fold bar construction $B^n(A)$ correspond to sums of planar fully grown trees with leaves labeled by elements in $ A$, see \cite{FrApp}. We fix some terminology concerning trees.

\begin{defn}
A planar fully grown $n$-level tree $t$ is a sequence
$$t = \tree$$
of order-preserving surjections. The element  $i \in [r_j]$ is called the $i$th vertex of the $j$th level, the elements in $[r_n]$ are also called leaves.
The degree of a tree $t$ is given by the number of its edges, \ie by
$$d(t) = \sum_{j=1}^n (r_j+1).$$
\end{defn}

\begin{defn}
For a given vertex $i \in [r_j]$ the subtree $t_{j,i}$ is the $(n-j)$-level subtree of $t$ with root $i$, \ie
$$
t_{j,i} = \xymatrix{ [\vert f_n^{-1}...f_{j+1}^{-1}(i)\vert -1] \ar[r]^-{g_n}& [\vert f_{n-1}^{-1}...f_{j+1}^{-1}(i)\vert -1] \ar[r]^-{g_{n-1}} & ... \ar[r]^-{g_{j+2}} & [\vert f_{j+1}^{-1}(i)\vert -1]}
$$
with $g_l$ the map making the diagram
$$\xymatrix{[\vert f_{l}^{-1}...f_{j+1}^{-1}(i)\vert -1] \ar[d]^-{\cong} \ar[r]^-{g_l} & [\vert f_{l-1}^{-1}...f_{j+1}^{-1}(i)\vert -1] \ar[d]^-{\cong}\\
f_{l}^{-1}...f_{j+1}^{-1}(i) \ar[r]^-{f_l} & f_{l-1}^{-1}...f_{j+1}^{-1}(i)
}$$
commute. Here the vertical maps are the unique order-preserving bijections.
\end{defn}

\begin{defn}[{\cite[Definition 3.1]{LR11}}] 
The category $\Epin$ has as objects planar fully grown trees with $n$ levels.
A morphism from $\treeR$ to $\treeS$ consists of surjections $h_i \colon [r_i ] \ra [s_i], 1\leq i \leq n$, such that the diagram
$$\xymatrix{
[r_n] \ar[r]^-{f_n^r} \ar[d]^-{h_n}& [r_{n-1}] \ar[r]^-{f_{n-1}^r} \ar[d]^-{h_{n-1}}& ... \ar[r]^-{f_2^r} & [r_1] \ar[d]^-{h_1}\\
[s_n] \ar[r]^-{f_n^s} & [s_{n-1}] \ar[r]^-{f_{n-1}^s} & ... \ar[r]^-{f_2^s} & [s_1]
}$$
commutes and such that $h_i$ is order-preserving on the fibres $(f^r_i)^{-1}(l)$ of $f^r_{i}$ for all $l \in [r_{i}]$. For $i=1$ we require that the map $h_1$ is order-preserving on $[r_1]$. 
The composite of two morphisms $(g_n,...,g_1) \colon t^q \ra t^r$ and $(h_n,...,h_1) \colon t^r \ra t^s$  is given by  $(h_n g_n,...,h_1 g_1).$
\end{defn}

Observe that since $A$ is concentrated in degree zero, the degree of a labeled tree viewed as an element in $B^n(A)$ is given by the number of edges of the tree. Lemma 3.5 in \cite{LR11} says that the maps in $\Epin$ decreasing the number of edges by one are exactly the summands of the differential of $B^n(A)$. This motivates the following definition.

\begin{defn}[{\cite[Definition 3.7]{LR11}}]\label{def:MulticomplexTrivial}
Let $F\colon \Epin \ra \kmod$ be a covariant functor. Let $\tilde{C}^{E_n}(F)$ be the $(\mathbb{N}\cup \lbrace0\rbrace)^n$-graded $k$-module
with 
$$\tilde{C}^{E_n}_{(r_n,...,r_1)} (F) = \bigoplus_{t=\tree} F(t).$$
For $1\leq j \leq n$ let $\tilde{\partial_j} \colon \tilde{C}^{E_n}\ra \tilde{C}^{E_n}$ be the following map lowering the $j$th degree by one:
\begin{itemize} 
\item Let $d_i \colon [r_n] \ra [r_{n}-1]$ be the order-preserving surjection which maps $i$ and $i+1$ to $i$. For $j=n$ define $\tilde{\partial_j}$ restricted to $F(t)$ as
$$\sum_{\substack{0 \leq i < r_n,\\ f_n(i) = f_n (i+1)}} (-1)^{s_{n,i}} F(d_i, \id_{[r_{n-1}]},...,\id_{[r_1]}).$$
\item
Let $1\leq j <n$, $0 \leq i <r_j$ and $\sigma \in \sh(f_{j+1}^{-1}(i), f_{j+1}^{-1}(i+1))$. Let $h=h_{i, \sigma}$ be the unique morphism of trees exhibited in \cite[Lemma 3.5]{LR11} with $h_j=d_i \colon [r_j] \ra [r_{j}-1]$, $h_l=\id$ for $l<j$ and $h_{j+1}$ restricted to $f_{j+1}^{-1}(\lbrace i, i+1 \rbrace)$ acting like $\sigma$.
Then $\tilde{\partial_j}$ is the map which restricted to $F(t)$ equals 
$$\sum_{\substack{0 \leq i < r_j,\\ f_j(i) = f_j(i+1)} }\sum_{\sigma \in \sh(f_{j+1}^{-1}(i), f_{j+1}^{-1}(i+1))}\epsilon(\sigma; t_{j,i}, t_{j, i+1})(-1)^{s_{j,i}} F(h_{i,\sigma}).$$
\end{itemize}

The signs arise from switching the degree $-1$ map $d_i$ with suspensions, as well as from the graded signature of the permutation $\sigma$ in the cases $j<n$.
More precisely, we number the edges in the tree $t$ from bottom to top and from left to right. For example the $2$-level tree $\xymatrix{[2] \ar[r]^-{f_2} & [1]}$ with $f_2(0)=f_2(1) =0$ and $f_2(2)=1$ is decorated as indicated in the following picture.
$$
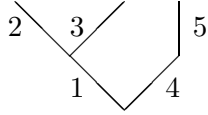
\setlength{\unitlength}{0.18cm}
\begin{picture}(8, 10)
\put(3, 5){\line(1, -1){4}}
\put(11, 5){\line(-1, -1){4}}

\put(3, 5){\line(-1, 1){4}}
\put(3, 5){\line(1, 1){4}}
\put(11, 5){\line(0, 1){4}}

\put(3,2){$1$}
\put(-1.5,6.5){$2$}
\put(3,6.5){$3$}
\put(10,2){$4$}
\put(12,6.5){$5$}

\end{picture}
$$
Then for $j<n$ we aquire a sign $(-1)^{s_{j,i}}$ where $s_{j,i}$ is the number of the rightmost top edge of the $(n-j)$-level subtree $t_{j,i}$ of $t$. For $j=n$ set $s_{n,i}$ to be the label of the edge whose leaf is the $i$th leaf for $0\leq i \leq n$.

For $j<n$ the map $F(h_{i, \sigma})$ is not only decorated by $(-1)^{s_{j,i}}$ but also by a  sign associated to $\sigma \in \sh(f_{j+1}^{-1}(i), f_{j+1}^{-1}(i+1))$: Let $t_1,...,t_a$ be the $(n-j-1)$-level subtrees of $t$ above the $j$-level vertex $i$, \ie the $(n-j-1)$-level subtrees forming $t_{j,i}$. Similarly let $t_{a+1},...,t_{a+b}$ denote the $(n-j-1)$-level subtrees above $i+1$. Then $\sigma$ determines a shuffle of $\lbrace t_1,...,t_a \rbrace$ and $ \lbrace t_{a+1},...,t_{a+b} \rbrace$. The sign $\epsilon( \sigma;t_{j,i}, t_{j,i+1})$  picks up a factor $(-1)^{(d( t_x) +1)(d(t_y) +1)}$ whenever $x<y$ and $\sigma(x) > \sigma(y)$.
\end{defn}

\begin{lemma}
For any functor $F \colon \Epin \ra \kmod$ the $(\mathbb{N} \cup \lbrace0\rbrace)^n$-graded module $\tilde{C}^{E_n}(F)$ together with $\tilde{\partial_1},...,\tilde{\partial_n}$ forms a multicomplex, which we again denote by $\tilde{C}^{E_n}(F)$.
\end{lemma}

\begin{defn}[{\cite[Definition 3.7]{LR11}}]\label{def:EnHomologyTrivial}
The homology 
$$H^{E_n}_*(F) = H_*(\Tot (\tilde{C}^{E_n}(F)))$$
of the total complex associated to $\tilde{C}^{E_n}(F)$ is called the $E_n$-homology of $F\colon \Epin \ra \kmod$.
\end{defn}

Livernet and Richter show that there is a Loday functor 
$$\mcL(A;k) \colon \Epin \ra \kmod$$ 
associated to every nonunital commutative algebra $A$ such that 
$$H_*^{E_n}(\mcL(A;k)) = H_*^{E_n}(A;k)$$ 
whenever $A$ is $k$-projective.
They then prove that $E_n$-homology of functors is indeed functor homology:

\begin{thm}[{\cite[4.1]{LR11}}]
Let $\tilde{b}\colon \Epin^{\op} \ra \kmod$ be the functor given by 
$$\tilde{b}(t) = \begin{cases}
k, & t= [0] \ra ... \ra [0], \\
0 & \text{else}.
\end{cases}$$
Then for $F \colon \Epin \ra \kmod$ 
$$H_*^{E_n}(F) = \Tor^{\Epin}_*(\tilde{b},F).$$
\end{thm}


\section{$E_n$-homology with coefficients via the iterated bar complex}

Recent work by Fresse and the author (see \cite{FZ}) shows that, at least for a commutative nonunital $k$-algebra $A$ and a symmetric $A$-bimodule $M$, the iterated bar complex can also be used to calculate $E_n$-homology and -cohomology with coefficients. 
In order to incorporate the action of $A$ on $M$ one has to add a twisting cochain 
$$\delta \colon A_+ \otimes B^n(A) \ra A_+ \otimes B^n(A)$$ to the complex $A_+ \otimes B^n(A)$.

\begin{defn}
Let $t(a_0,...,a_{r_n})$ denote the element in $B^n(A)$ defined by the $n$-level tree $t= \tree$ with leaves labeled by $a_0,...,a_{r_n} \in A$. The twisting morphism $\delta \colon A_+ \otimes B^n(A) \ra A_+ \otimes B^n(A)$ is given by
\begin{eqnarray*}
\delta(a \otimes t(a_0,...,a_{r_n})) & = & \sum_{\substack{0\leq l \leq r_{n-1},\\ \vert f_n^{-1}(l)\vert >1,\\ x=\min f_n^{-1}(l)}} (-1)^{s_{n,x}-1} a a_x \otimes (t \setminus x)(a_0,...,\hat a_x,...,a_{r_n})  \\
&+& \sum_{\substack{0 \leq l \leq r_{n-1}\\ \vert f_{n}^{-1}(l)\vert >1,\\ y=\max f_n^{-1}(l)}} (-1)^{s_{n,y}}a_y a\otimes (t \setminus y)(a_0,...,\hat a_y,...a_{r_n})
\end{eqnarray*}
for $a\in A_+$.
Here for $s \in [r_n]$ such that $s$ is not the only element in the corresponding $1$-fibre of $t$ containing $s$, we let $t\setminus s$ be the tree obtained by deleting the leaf $s$. To be more precise, 
$$t\setminus s = \xymatrix{[r_n-1] \ar[r]^-{f_n'} & [r_{n-1}] \ar[r]^-{f_{n-1}} & ... \ar[r]^-{f_2}  & [r_1]}$$
with 
$$f_n'(x) = \begin{cases}
f_n(x), & x< s,\\
f_n(x+1), & x\geq s.
\end{cases}$$
The sign $(-1)^{s_{n,i}}$ is as in Definition \ref{def:MulticomplexTrivial}.
\end{defn}

\begin{remark}
\begin{enumerate}
\item
Intuitively the map $\delta$ deletes leaves and acts with the corresponding label on the coefficient module $A_+$. The leaves which are deleted are either on the left or on the right of a $1$-fibre of the tree. For $n=1$ compare this to the complex calculating Hochschild homology $HH(A;A_+)$:  
The standard differential maps $a \otimes a_0 \otimes ... \otimes a_l\in A_+ \otimes A^{\otimes l+1}$ to
\begin{eqnarray*}
&&a a_0 \otimes a_1 \otimes ... \otimes a_l + (-1)^{l+1}a_l a \otimes a_0 \otimes ... \otimes a_{l-1}\\
& +& \sum_{i=0}^{l-1} (-1)^{i+1}a \otimes a_0 \otimes ... \otimes a_i a_{i+1} \otimes ... \otimes a_l.
\end{eqnarray*} 
The first two summands correspond to the twist $\delta$, while the other summands correspond to $\partial_B$.
\item The map $\delta$ only considers $1$-fibres of arity at least two because on $1$-fibres of arity one the two summands in the definition cancel each other out: Since $A_+$ is commutative, multiplying $a\in A_+$ with $a_i \in A$ from the left equals multiplying with $a_i$ from the right.
\end{enumerate}
\end{remark}

In section \ref{sec:Epinp} we will define $E_n$-homology and $E_n$-cohomology of functors defined on a category which extends the category $\Epin$. The following theorem will allow us to argue in Remark \ref{rem:LodayHom} and Remark \ref{rem:LodayCo} that $E_n$-homology and $E_n$-cohomology of functors encompass $E_n$-homology and $E_n$-cohomology of commutative algebras with coefficients in a symmetric bimodule.

\begin{thm} \cite{FZ}\label{thm:FZ}
For a commutative $k$-projective nonunital $k$-algebra $A$ and a symmetric $A$-bimodule $M$ we have
$$H_*^{E_n}(A;M) = H_*(M \otimes_{A_+} (A_+ \otimes \Sigma^{-n}B^n(A), \delta))
$$
and
$$H^*_{E_n}(A;M) = H^*(\Hom_{A_+}((A_+ \otimes \Sigma^{-n} B^n(A),\delta), M).$$
\end{thm}


\section{The category $\Epinp$ encoding the $n$-fold bar complex with coefficients}\label{sec:Epinp}

We would like to establish a functor homology interpretation for $E_n$-homology of a commutative algebra $A$ with coefficients in a symmetric $A$-bimodule $M$ as well as for $E_n$-cohomology.  
To model $E_n$-homology with coefficients as functor homology we  have to enlarge the category $\Epin$ to incorporate the summands of the twisting cochain $\delta$.

\begin{defn}\label{defn:Epinp}
The objects of the category $\Epinp$ are given by planar fully grown trees with $n$ levels. A  morphism from $t^r=\treeR$ to $t^s=\treeS$ is represented by a sequence of maps $(h_n,...,h_1)$, where
\begin{itemize}
\item  for $i=2,...,n-1$, the map $h_i\colon [r_i]\rightarrow [s_i]$ is a  surjection which is order-preserving on the fibres $f_i^{-1} (l)$ for all $l \in [r_{i-1}]$. For $i=1$ we require $h_1 \colon [r_1 ]\ra [s_1]$ to be order-preserving. 
\item The map $h_n$ is a map 
\[
h_n\colon [r_n] \rightarrow [s_n]_+ := [s_n] \sqcup \lbrace + \rbrace
\]
such that $[s_n]$ lies in the image of $h_n$. We also require that the restriction of $h_n$ to $h_n^{-1}([s_n])$ is order-preserving on the fibres of $f_n$. Furthermore the intersection of  $h_n^{-1}([s_n])$ with a fibre $f_n^{-1}(l)$ has to be a (potentially empty) interval for all $l \in [r_{n-1}]$, \ie is of the form $\lbrace a, a+1,..., a+b \rbrace$ with $b\geq -1$. \\
\item The diagram 
\[
\xymatrix{
h_n^{-1}([s_n]) \ar[r]^{f^r_n} \ar[d]^{h_n} & [r_{n-1}]\ar[d]^{h_{n-1}} \ar[r] & ... \ar[r] & [r_2] \ar[d]^{h_2}\ar[r]^{f^r_1} &[r_1] \ar[d]^{h_1}\\
[s_n] \ar[r]^{f^s_n} & [s_{n-1}] \ar[r] & ... \ar[r] & [s_2] \ar[r]^{f^s_1} &[s_1]
}
\]
commutes.
\end{itemize}
Finally we identify certain morphisms by imposing the following equivalence relation on the set of morphisms from $t^r$ to $t^s$: We identify morphisms $h$ and $h'$ if
\begin{itemize}
\item ${h_n}^{-1}( +) = {h'_n}^{-1} (+)$ and
\item for all $1\leq i \leq n$ the maps $h_i $ and  $h'_i$ coincide if restricted to $f_{i+1}^r ... f_n^r([r_n]\setminus h_n^{-1}(+))$.
\end{itemize}

The composition of two morphism $(g_n,...,g_1)\colon t^q \ra t^r$ and $(h_n,...,h_1)\colon t^r \ra t^s$ is defined by composing componentwise and sending $+$ to $+$, \ie
$$(h_n,...,h_1) \circ (g_n,...,g_1) := ((hg)_n, h_{n-1} g_{n-1},...,h_1 g_1)$$
with $(hg)_n(x) = \begin{cases} +, & g_n(x) = +,\\ h_ng_n(x) & \text{otherwise.} \end{cases}$
\end{defn}

A straightforward calculation shows that composition in $\Epinp$ is well defined and associative.

\begin{remark} \label{rem:added morphisms}
\begin{enumerate}
\item It is clear that $\Epin$ is a subcategory of $\Epinp$ and that both categories share the same objects. Let $\delta_i\colon[r_n]\ra [r_{n}-1]$ be the map 
$$\delta_i(x)= \begin{cases} x, & x <i,\\
+, & x=i,\\
x-1, & x>i.
\end{cases}$$
Intuitively the category $\Epinp$  is built from $\Epin$ by adding morphisms of the form
$$\xymatrix{
[r_n]\ar[r]^{f_n}\ar[d]_{\delta_i}& [r_{n-1}] \ar[d]_{\id} \ar[r]^{f_{n-1}}&... \ar[r]^{f_2} & [r_1]\ar[d]_{\id}\\
[r_n-1] \ar[r]^{\hat{f}_n} & [r_{n-1}] \ar[r]^{f_{n-1}} &... \ar[r]^{f_2} & [r_1]
}$$
with  $i$ the minimal or maximal element of a fibre $f_n^{-1}(l)$ containing at least two elements. Here 
$$\hat{f}_n(x) = \begin{cases} f_n(x), & x <i,\\
f_n(x+1), & x\geq i.
\end{cases}$$
The requirement that the elements of a fibre of $f_n$ that are not mapped to $+$ form an interval reflects the fact that we have only added morphisms of the aforementioned kind.
\item We only added morphism $(\delta_i, \id,...,\id) \colon t \ra t'$ such that $i$ is not the only element in the corresponding $1$-fibre of $t$. Nevertheless it is possible to map $1$-fibres of arity one to $+$ by first applying maps which merge edges in lower levels. For example, the map 
$$
\xymatrix{[1] \ar[r]^-{\id} \ar[d]_-{\substack{0 \mapsto +,\\ 1 \mapsto 0}} & [1] \ar[d]^-{d_0} \\
[0]\ar[r]^-{\id} &[0]}
$$
arises as the composite of the maps
$$
\xymatrix{[1] \ar[r]^-{\id} \ar[d]_-{\id} & [1] \ar[d]^-{d_0} \\
[1]\ar[r]^-{0,1 \mapsto 0} &[0]}
\quad \quad \text{and} \quad \quad
\xymatrix{[1] \ar[r]^-{0,1 \mapsto 0} \ar[d]_-{\delta_0} & [0] \ar[d]^-{\id} \\
[0]\ar[r]^-{\id} &[0]}
$$
\item The motivation for defining $\Epinp$ is to model the complex calculating $E_n$-homology of $A$ with coefficients in $M$. Hence imposing the above equivalence relation on the set of morphisms is necessary: It should not matter what precisely happens to a subtree of a tree $t$ if all its leaves get mapped to $+$, \ie in which order and on what side of an element we act on with  a family of elements of  $A$.
\end{enumerate}
\end{remark}

After defining the category $\Epinp$ which also models the summands of the twisting cochain $\delta$, we can proceed to define $E_n$-homology of a functor.

\begin{defn}\label{def:FunktorHomWithCoeffsDiffs}
Let $F\colon \Epinp \rightarrow \kmod$ be a functor. As in Definition \ref{def:MulticomplexTrivial} set 
$$C^{E_n}_{r_n,...,r_1}(F):=  \bigoplus_{t=[r_n] \rightarrow ... \rightarrow [r_1]} F(t).$$
Define maps $\partial_j\colon C^{E_n}_{r_n,...,r_j,...,r_1} \ra C^{E_n}_{r_n,...,r_j-1,...,r_1}$ lowering the $j$th degree by one
by
$$\partial_j = \tilde{\partial}_j \quad \text{for } i<n \quad \quad\text{and} \quad \quad
\partial_n = \tilde{\partial}_n + \delta_{\min} + \delta_{\max},$$ 
with 
$$\delta_{\min}= \sum_{\substack{0\leq l \leq r_{n-1},\\ 
\vert f_n^{-1}(l)\vert >1}} (-1)^{s_{n,\min f_n^{-1}(l)}-1}F(\delta_{\min f_n^{-1}(l)}, \id,...,\id) $$
and
$$\delta_{\max}= \sum_{\substack{0 \leq l \leq r_{n-1}\\ \vert f_{n}^{-1}(l)\vert >1}} (-1)^{s_{n,\max f_n^{-1}(l)}}F(\delta_{\max f_n^{-1}(l)}, \id,...,\id).$$
The integers $s_{n,i}$ are as in Definition \ref{def:MulticomplexTrivial}.

\end{defn}

\begin{example}
Let $t$ be the $2$-level tree

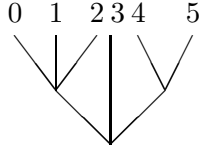
\[
\setlength{\unitlength}{0.18cm}
\begin{picture}(8, 10)
\put(3, 5){\line(1, -1){4}}
\put(7, 5){\line(0, -1){4}}
\put(11, 5){\line(-1, -1){4}}

\put(0, 9){\line(3, -4){3}}
\put(3, 9){\line(0, -1){4}}
\put(6, 9){\line(-3, -4){3}}
\put(-0.5, 10){$0 $}
\put(2.5, 10){$1$}
\put(5.5, 10){$2$}

\put(7, 5){\line(0, 1){4}}
\put(7, 10){$3$}

\put(9, 9){\line(1, -2){2}}
\put(13, 9){\line(-1, -2){2}}
\put(8.5, 10){$4 $}
\put(12.5, 10){$5$}

\end{picture}
\]
Then $\delta_{\min}$ is the sum of the morphism induced by mapping the leaf labeled $0$ to $+$, equipped with the sign $(-1)^1$, and the morphism induced by mapping $4$ to $+$, decorated by $(-1)^7$. The map $\delta_{max}$ is induced by sending $2$ to $+$ with sign $(-1)^4$ and by mapping $5$ to $+$ which yields the sign $(-1)^9$.
\end{example}

We already know from \cite[Lemma 3.8]{LR11} that $(C^{E_n}, \tilde{\partial}_1,...,\tilde{\partial}_n)$ is a multicomplex. Hence it suffices to prove the following lemma, which can be done via a tedious, but straightforward calculation, see \cite[Lemma 4.14]{Z}.

\begin{lemma}\label{lemma:checkchaincomplex}
Let $F\colon \Epinp \ra \kmod$. The maps defined above satisfy the identities
\begin{eqnarray*}
(\delta_{\min} + \delta_{\max}) \partial_j + \partial_j (\delta_{\min} + \delta_{\max})&= &0 \quad \text{for all }j<n,\\
(\delta_{\min}+\delta_{\max})^2 + \tilde{\partial}_n (\delta_{\min} +\delta_{\max}) + \tilde{\partial}_n(\delta_{\min}+\delta_{\max})&= &0.
\end{eqnarray*}
Hence $C^{E_n}(F)$ is a multicomplex. 
\end{lemma}

\begin{defn}
Let $F \colon \Epinp \ra \kmod$ be a functor. The $E_n$-homology of $F$ is
$$H^{E_n}_*(F) = H_*(\Tot (C^{E_n}(F))).$$
\end{defn}

\begin{remark}
Given a functor $\tilde F \colon \Epin \ra \kmod$, we can extend $\tilde F$ to $F \colon \Epinp \ra \kmod$ by setting $F(h) = 0$ for every morphism $h \colon t^r \ra t^s$ in $\Epinp$ such that $h([r_n]) \cap \lbrace + \rbrace \neq \emptyset.$ With these definitions $H^{E_n}(F)$ coincides with the $E_n$-homology of $\tilde F$ as defined in Definition \ref{def:EnHomologyTrivial}. In this sense the definition of $E_n$-homology we just gave extends the definition given in \cite[Definition 3.7]{LR11}.
\end{remark}

We are specifically interested in calculating $E_n$-homology of commutative algebras, which is the $E_n$-homology of the following functors. 

\begin{remark}\label{rem:LodayHom}
The Loday functor $\mcL (A;M)\colon \Epinp \rightarrow \kmod$ 
is the following functor: For a given tree $t=\tree$ set
$$ \mcL(A;M) (t) = M \otimes A^{\otimes r_n+1}.$$
If $(h_n,...,h_1) \colon t^r \ra t^s$ is a morphism 
let 
$$\mcL(A;M)(h_n,...,h_1) \colon  M\otimes A^{\otimes r_n +1}  \rightarrow M \otimes A^{\otimes s_n+1} $$ 
be given by
$$  m \otimes a_0 \otimes...\otimes a_{r_n}  \mapsto   \left( m \cdot \prod_{\substack{i\in [r_n] \\  h_n(i) = +}} a_i \right) \otimes \left( \prod_{\substack{i\in [r_n] \\  h_n(i) = 0}} a_i \right) \otimes... \otimes \left(\prod_{\substack{i\in [r_n] \\  h_n(i) = s_n}} a_i \right) .$$
Then $\Tot( C^{E_n}(\mcL (A;M))) = \Sigma^{-n} (M\otimes_{A_+}(A_+ \otimes B^n(A), \delta))$. In particular, by Theorem \ref{thm:FZ} 
$$H_*^{E_n}(\mcL(A;M)) = H_*^{E_n}(A;M)$$
if $A$ is $k$-projective. Note that $\mcL(A;k)$ agrees with the extension of the Loday functor defined by Livernet and Richter in \cite[3.1]{LR11} to $\Epinp$.
\end{remark}

We now consider $E_n$-cohomology. The definition of $E_n$-cohomology is dual to the definition of $E_n$-homology.

\begin{defn} 
Let $G\colon {\Epinp}^{\op} \rightarrow \kmod$ be a functor. 
The $E_n$-cohomology of $G$ is defined as
$$ H^*_{E_n} (G) = H_*(\Tot (C_{E_n}(G))),$$
with the multicomplex $C_{E_n} (G)$ defined as follows:
We set
$$C_{E_n}^{r_n,...,r_1}(G) =  \bigoplus_{t=\tree} G(t).$$
The differentials $\partial_j\colon C_{E_n}^{r_n,...,r_1}(G) \rightarrow C_{E_n}^{r_n,...r_j+1,...,r_1}(G)$ raise the $j$th degree by one. 
For $j=n$ define $\partial_n$ restricted to $G(t)$ as
\begin{eqnarray*}
&& \sum_{\substack{0 \leq i < r_n,\\ f_n(i) = f_n (i+1)}} (-1)^{s_{n,i}} G(d_i, \id_{[r_{n-1}]},...,\id_{[r_1]}) \\
&+& 
 \sum_{\substack{0\leq l \leq r_{n-1},\\ 
\vert f_n^{-1}(l)\vert >1}} (-1)^{s_{n,\min f_{n}^{-1}(l)}-1}G(\delta_{\min f_n^{-1}(l)}, \id,...,\id) \\
& + &
\sum_{\substack{0 \leq l \leq r_{n-1}\\ \vert f_{n}^{-1}(l)\vert >1}} (-1)^{s_{n,\max f_n^{-1}(l)}}G(\delta_{\max f_n^{-1}(l)}, \id,...,\id).\end{eqnarray*}
For $1\leq j <n$ the map $\partial_j$ restricted to $G(t)$ is given by
$$\sum_{\substack{0 \leq i < r_j,\\ f_j(i) = f_j(i+1)} }\sum_{\sigma \in \sh(f_{j+1}^{-1}(i), f_{j+1}^{-1}(i+1))}\epsilon(\sigma; t_{j,i},t_{j,i+1})(-1)^{s_{j,i}} G(h_{i,\sigma}).$$
Here $h=h_{i, \sigma}$ again denotes the unique morphism of trees exhibited in \cite[Lemma 3.5]{LR11} with $h_j=d_i \colon [r_j] \ra [r_{j}-1]$, $h_l=\id$ for $l<j$ and $h_{j+1}$ restricted to $f_{j+1}^{-1}(\lbrace i, i+1 \rbrace)$ acting like $\sigma$.
\end{defn}

As was the case for $E_n$-homology this definition generalizes $E_n$-cohomology of commutative algebras with coefficients in a symmetric bimodule:

\begin{remark}\label{rem:LodayCo}
Let $\mcL^{c}(A;M)\colon {\Epinp}^{\op} \rightarrow \kmod$
be defined on  $t=\tree$ as
$$ \mcL^{c}(A;M)(t) =  \Hom_k(A^{\otimes r_n+1}, M).$$
If $(h_n,...,h_1)$ is a morphism from $t^r$ to $t^s$ 
define  
$$ \mcL^{c}(A;M)(h_n,...,h_1) \colon \Hom_k(  A^{\otimes s_n +1},M)  \rightarrow \Hom_k(  A^{\otimes r_n+1},M)$$ 
 by
$$( \mcL^{c}(A;M)(h_n,...,h_1)(f)) (a_0 \otimes ... \otimes a_{r_n})=  \left( \prod_{\substack{i\in [r_n] \\  h_n(i) = +}} a_i \right) \cdot f\left(\left( \prod_{\substack{i\in [r_n] \\  h_n(i) = 0}} a_i \right) \otimes... \otimes \left( \prod_{\substack{i\in [r_n] \\  h_n(i) = s_n}} a_i\right) \right).$$
Then $\Tot(C_{E_n}(\mcL^c(A;M)))$ coincides with the complex computing $E_n$-cohomology of $A$ with coefficients in $M$. Theorem \ref{thm:FZ} hence yields that
$$H^*_{E_n}(\mcL^c(A;M)) = H^*_{E_n}(A;M)$$
if $A$ is $k$-projective. 
\end{remark}


\section{$E_n$-cohomology as functor cohomology}

In \cite[Theorem 4.1]{LR11} Livernet and Richter show that $E_n$-homology with trivial coefficients can be interpreted as functor homology. We now extend this result to $E_n$-homology and $E_n$-cohomology with arbitrary coefficients. Like in \cite{LR11} we prove that $E_n$-homology coincides with functor homology by using the axiomatic characterizations of $\Tor$ and $\Ext$. For a background on functor homology we refer the reader to \cite{PR02}. We first show that certain projective functors are acyclic. Recall that for a small category $\mathcal{C}$ a functor $F\colon \mathcal{C} \ra \kmod$ is called projective if it has the usual lifting property with respect to objectwise surjective natural transformations. For $t \in \Epinp$ define projective functors $P_t$ and $P^t$ by
$$P_t = k[ \Epinp(t,-) ] \colon \Epinp \ra \kmod \quad \text{and} \quad P^t=  k[ \Epinp(-,t)] \colon {\Epinp}^{\op} \ra \kmod.$$

In the proof of the following lemma, we will consider trees obtained by restricting a given tree to certain leaves.

\begin{defn}\label{defn:restrictedTree}
Let $t=\tree$ be a tree. For fixed $I \subset [r_n]$ set $r^I_i = \vert f_n ... f_{i+1}(I) \vert-1$. Define a tree $t^I$ as the upper row in 
$$\xymatrix{
[r_n^I ] \ar[d]  \ar[r]^-{f_n^I} & [r_{n-1}^I ] \ar@<1ex>[d] \ar[r]^-{f_{n-1}^I} & ... \ar[r]^-{f_2^I} & [r_1^I]  \ar[d]  \\
I \ar[r]^-{f_n} & f_n(I) \ar[r]^-{f_{n-1}}& ... \ar[r]^-{f_2}& f_2...f_n(I) 
}$$
Here the vertical morphisms are determined by requiring that they are bijective and order-preserving, while the maps $f_n^I$ are defined by requiring that all squares commute.
Intuitively $t^I$ is the subtree of $t$ given by restricting $t$ to edges connecting leaves labeled by $I$ with the root. 
\end{defn}

\begin{lemma}\label{lem:MorphismDeletingLeaves}
 Let $t=\tree$ be a tree. Let $I\subset[r_n]$ be a set such that $I \cap f_n^{-1}(i)$ is a (possibly empty) interval for all $i \in [r_{n-1}]$. Then we can define a morphism $h^I \colon t \ra t^I$ in $\Epinp$ as the vertical maps in
$$\xymatrix{
[r_n] \ar[d]^{h_n^I} \ar[r]^-{f_n} & [r_{n-1}] \ar[d]^{h_{n-1}^I} \ar[r]^-{f_{n-1}} & ... \ar[r]^-{f_2} & [r_1] \ar[d]^{h_1^I}\\
[r_n^I] \ar[r]^-{f_n^I} & [r_{n-1}^I] \ar[r]^-{f_{n-1}^I} &... \ar[r]^-{f_2^I} & [r_1^I].
}$$
These are defined as follows: The map $h_n^I$ maps all $x \in [r_n]\setminus I$ to $+$ and is an order-preserving bijection restricted to $I$. For $i<n$ we require that $h_i^I$ restricted to $f_{i+1}...f_n(I)$ is the order-preserving bijection to $[r_i^I]$ and that $h_i^I$ be order-preserving on the whole set $[r_i]$. 
\end{lemma}
\begin{proof}
Recall that a morphism in $\Epinp$ is an equivalence class with respect to the equivalence relation introduced in Definition \ref{defn:Epinp}.  Since $I= [r_n] \setminus (h_n^I)^{-1}(+) $ the above requirements uniquely determine $h^I$ up to equivalence. The maps $h^I_i$ assemble to a morphism in $\Epinp$ since they are chosen to be order-preserving and the squares 
$$\xymatrix{
f_{i+1}...f_n(I) \ar[r]^-{f_i} \ar[d]^{h_i^I} & f_i...f_n(I) \ar[d]^{h_{i-1}^I}\\
[r_i^I] \ar[r]^-{f_i^I} & [r_{i-1}^I]
}$$
commute by definition of $f_i^I$. Furthermore  $(h_n^I)^{-1}(+) \cap f_n^{-1}(i) = I \cap f_n^{-1}(i)$ is an interval.
\end{proof}

Now we are in the position to compute the $E_n$-homology of the representable projectives.

\begin{lemma}
Fix a tree $t=\tree$. Then
$$H_*^{E_n}(P_t) = \begin{cases} 0, & *>0,\\
\bigoplus_{i \in [r_n]} k, & *=0.\end{cases}$$
\end{lemma}
\begin{proof}
Set $C:= \Tot(C^{E_n}(P_t))$. We define an ascending filtration by subcomplexes of $C$ by 
$$F^pC_{s_n,...,s_1}:= \bigoplus_{t^s= \treeS} k[\lbrace (h_n,...,h_1) \in P_t( t^s) : \vert h_n^{-1}([s_n])\vert \leq p+1 \rbrace].$$ 
Hence $F^pC$ is  generated by morphisms that map at least $r_n-p$ leaves to $+$. This yields a first quadrant spectral sequence
$$ E^1_{p,q}= H_{p+q}(F^pC / F^{p-1}C) \Rightarrow H_{p+q}( C).$$ 
The quotient $F^p C / F^{p-1} C$ can be identified with the free $k$-module generated by  morphisms $(h_n,...,h_1) \in k[\Epinp(t,t^s)]$ with $\vert h_n^{-1}([s_n])\vert = p+1$. The differentials $\delta_{\min}$ and $\delta_{\max}$ vanish on this quotient.
 The remaining summands of $\partial_n$ and the differentials $\partial_{n-1},...,\partial_1$ do not change the number of leaves that get mapped to $+$.
We conclude that $F^p C / F^{p-1}C$ is isomorphic to $D$ as a complex, where 
$$D_{s_n,...,s_1} = \bigoplus_{t^s=\treeS} k[\lbrace (h_n,...,h_1) \in P_t( t^s) : \vert h_n^{-1}([s_n])\vert = p+1 \rbrace]$$ 
with differentials $\partial_1,..., \partial_{n-1}$ and $\hat{\partial}_n= \partial_n-\delta_{\min} - \delta_{\max}$.
The complex $D$ can be decomposed further: The remaining differentials do not only respect the number of deleted leaves but also the set of deleted leaves itself. Hence $D$ is the direct sum of subcomplexes $D^I$ with
$$D^I_{s_n,...,s_1} = \bigoplus_{t^s= \treeS} k[\lbrace  (h_n,...,h_1) \in P_t( t^s): h_n^{-1}([s_n])=I \rbrace]$$  
such that $I$ is a subset of $[r_n]$ of cardinality $p+1$. 
Notice that the differentials of $D$ and $D^I$ look like the differentials used in Definition \ref{def:MulticomplexTrivial} to define $E_n$-homology of functors from $\Epin$ to $\kmod$. We will show that $D^I$ in fact can be identified with the complex associated to such a functor. More precisely, $D^I$ is the complex computing $E_n$-homology of the representable functor $k[\Epin(t^I,-)] \colon \Epin \ra \kmod$:
Denote by $h^I\colon t \ra t^I$ the morphism defined in Lemma \ref{lem:MorphismDeletingLeaves}. We define $$\Psi\colon \tilde{C}^{E_n}(\Epin(t^I,-)) \rightarrow D^I$$
by mapping $j \in \Epin(t^I, t^s)$ to $\Psi(j)= j \circ h^I$. Since $j$ does not delete any leaves this yields an element of $D^I$. 
We define an inverse $\Phi$ to $\Psi$ by mapping $h\in D^I$ to the composite of the columns in
$$\xymatrix{
[r_n^I]  \ar[r]^-{f_n^I} \ar[d] & [r_{n-1}^I] \ar[r]^-{f_{n-1}^I} \ar[d] &... \ar[r]^-{f_2^I} & [r_1^I] \ar[d]\\
I \ar[d] \ar[r]^-{f_n} & f_n(I) \ar[d] \ar[r]^-{f_{n-1}} & ... \ar[r]^-{f_2} & f_2...f_n(I) \ar[d]\\
[r_n] \ar[d]^{h_n} \ar[r]^-{f_n} & [r_{n-1}]\ar[d]^{h_{n-1}} \ar[r]^-{f_{n-1}} & ... \ar[r]^-{f_2} & [r_1] \ar[d]\\
[s_n] \ar[r]^-{g_n} & [s_{n-1}] \ar[r]^-{g_{n-1}} & ... \ar[r]^-{g_2} & [s_1]
}$$
Here the upper vertical maps are order-preserving bijections while the vertical maps in the middle are inclusions. We see that $\Phi(h)_i$ only depends on ${h_i}_{\vert {f_{i+1}...f_n(I)}}$, \ie $\Phi$ is well defined on equivalence classes. It is obvious that each $\Phi(h)_i$ is surjective and that the usual requirements on commutativity are satisfied.
Consider a  fibre $(f_i^I)^{-1}(l)$: The map $\Phi(h)_i$ first sends it order-preservingly and surjectively to $f_{i+1}...f_n(I)\cap f_i^{-1}(l')\subset [r_i]$, where $l'$ denotes the image of $l$ under the map $[r_{i-1}^I] \rightarrow f_i...f_n(I)$. 
Since $h_i$  preserves the order on fibres of $f_i$ we see that $\Phi(h)_i$ is order-preserving on the fibres of $f_i^I$. 
Hence $\Phi$ is indeed a map from $D^I$ to $\tilde{C}^{E_n}(\Epi_n(t^I,-))$.

Finally we note that obviously $\Phi \circ \Psi$ is the identity. To show that $\Psi$ is a left inverse for $\Phi$ one writes down $(\Psi \circ \Phi )(h)$ for a given $h$ and uses that $((\Psi \circ \Phi)(h))_i$ only needs to coincide with $h_i$ on $f_{i+1}...f_n (I)$.
The maps $\Phi$ and $\Psi$  commute with composition, hence also with applying the differentials. Since the signs in the differentials applied to a morphism $h$ are determined by the target tree $t^s$ of $h$, there is no trouble with signs either. Hence we have constructed an isomorphism
$$ D^I \cong C^{E_n}(\Epin(t^I,-))$$
of complexes. 

We know from \cite[Section 4]{LR11} that $H_*(\Tot (\tilde C^{E_n}(\Epin(t^I,-)))) =0$ for $*>0$ and that
$$H_0(\Tot( \tilde C^{E_n}(\Epin(t^I,-)))) = \begin{cases} k, & t^I= [0] \rightarrow [0] \rightarrow ... \rightarrow [0],\\
0, & \text{else.}
\end{cases}$$
Since $t^I=[0] \rightarrow [0] \rightarrow ... \rightarrow [0]$ implies $p+1=\vert I \vert =1$ we see that the $E^1$-term of our spectral sequence is
$$E^1_{p,q} = H_{p+q}(F^pC / F^{p-1}C) = \begin{cases} \bigoplus_{i \in [r_n]} k, & p=q=0,\\
0, & \text{else.}
\end{cases}$$
The spectral sequence collapses and the claim follows.
\end{proof}

Having proved that $H^{E_n}_*(P_t)$ is acyclic we can use the axiomatic decription of $\Tor$ (see \eg \cite[Ch.2]{Gr57}). 

\begin{thm}
Denote by $b\colon{\Epinp}^{\op} \ra \kmod$ the functor given by the cokernel of
$$ (\delta_0, \id,...,\id)_* - (d_0, \id, ..., \id)_*+ (\delta_1, \id,...,\id)_* \colon 
 P^{[1] \rightarrow  [0] \rightarrow ... \rightarrow [0]} \rightarrow P^{[0]  \rightarrow...\rightarrow [0]}.$$
Then for any $F\colon \Epinp \ra \kmod$
$$H_*^{E_n}(F) \cong \Tor^{\Epinp}_*(b,F),$$
and this isomorphism is natural in $F$.
\end{thm}
\begin{proof}
A short exact sequence $0 \rightarrow F \rightarrow G \rightarrow H \rightarrow 0$ of functors
yields a short exact sequence of chain complexes
$$0 \rightarrow \Tot (C^{E_n}(F)) \rightarrow\Tot( C^{E_n}(G)) \rightarrow \Tot(C^{E_n}(H)) \rightarrow 0.$$
This in turn gives rise to a long exact sequence on homology. 
We already showed that $H_*^{E_n}(P_t)$ is zero in positive degrees. Every projective functor from $\Epinp$ to $\kmod$ receives a surjection from a sum of functors of the form of $P_t$. It hence is a direct summand of this sum. Therefore $H_*^{E_n}(P)$ vanishes in positive degrees for all projective functors $P$.
Finally the zeroth $E_n$-homology of a functor $F$ is given by the cokernel of
$$
(-1)^{n-1}F(\delta_0, \id,...,\id) + (-1)^n F(d_0, \id, ..., \id)+ (-1)^{n+1}F(\delta_1, \id,...,\id) . 
$$
Using the natural isomorphism $P^t\otimes_{\Epinp} F  \cong F(t)$ of $k$-modules and that tensor products are right exact, one sees that this coincides with $b \otimes_{\Epinp} F$.
\end{proof}

Every functor $F \colon \Epinp \ra \kmod$ gives rise to a functor $F^*\colon{\Epinp}^{\op} \ra \kmod$, its dual, by setting $F^*(t)= \Hom_k(F(t),k)$. Since we just proved that $E_n$-homology of projective functors vanishes, we can relate  $E_n$-homology with $E_n$-cohomology via  the following Grothendieck spectral sequence . 

\begin{proposition}[see \eg {\cite[Theorem 10.49]{Ro09}}]
\label{prop:ukt}
If $F(t)$ is $k$-free for every $t\in \Epinp$, there is a first quadrant spectral sequence
$$E^2_{p,q} = \Ext_k^q(H^{E_n}_p(F),k) \Rightarrow H^{p+q}_{E_n}(F^*).$$
In particular whenever $k$ is injective as a $k$-module, $E_n$-homology of $F$ and $E_n$-cohomology of its dual are dual to each other.
\end{proposition}

Examples of commutative self-injective rings include fields, group algebras of finite commutative groups over a self-injective ring, quotients $R / I$ of a principal ideal domain $R$ with $I\neq 0$ and commutative Frobenius rings \cite[Ch.5, \S 18]{AF92}.
The product of self-injective rings is again self-injective.

\begin{thm}
Suppose that $k$ is injective as a $k$-module and let $G\colon {\Epinp}^{\op} \ra\kmod$ be a functor. Then there is an isomorphism
$$H^*_{E_n}(G) \cong \Ext_{\Epinp}^*(b,G).$$
This isomorphism is natural in $G$.
\end{thm}
\begin{proof}
That $H^*_{E_n}$ maps short exact sequences to long exact sequences follows as in the homological case. Since the projective functor $P_t$ is finitely generated and $k$-free, the functor $P_t^*$ is injective. The universal coefficient spectral sequence \ref{prop:ukt} yields that these modules are acyclic. But then all other injective modules are acyclic, too, since they are direct summands of products of these. 
Finally let $G\colon {\Epinp}^{\op} \ra \kmod$ be an arbitrary functor. Then the zeroth $E_n$-cohomology of $G$ is by definition the kernel of
\begin{eqnarray*}
(-1)^{n-1}G(\delta_0, \id,...,\id) + (-1)^n G(d_0, \id, ..., \id)+ (-1)^{n+1}G(\delta_1, \id,...,\id).
\end{eqnarray*}
The Yoneda lemma  and the left exactness of $\Nat_{\Epinp}(-,G)$ yield that this kernel results from applying $\Nat_{\Epinp}(-,G)$ to $b$.
\end{proof}


\section{Functor cohomology and cohomology operations}

We recall the definition of the Yoneda pairing on $\Ext$. The Yoneda pairing is usually defined  in the context of modules over a ring (see \eg \cite[III.5, III.6]{ML95}). But it is well known to be easily generalized to suitable abelian categories with enough projectives and injectives. We assume that $k$ is self-injective in this section. 
\begin{defn}\label{def:yoneda}
Let $F,G$ and $H$ be functors from ${\Epinp}^{\op}$ to $\kmod$. Let $P_F$ denote a projective resolution of $F$ and $I_H$ an injective resolution of $H$. There is a pairing
$$\mu \colon \Ext^*_{\Epinp}(G,H) \otimes \Ext^*_{\Epinp}(F,G) \rightarrow \Ext^*_{\Epinp}(F,H),$$
defined as the composite
$$\xymatrix{
 \Ext^m_{\Epinp}(G,H)  \otimes  \Ext^n_{\Epinp}(F,G) \ar@{=}[d]\\
 H_m(\Nat_{\Epinp}(G, I_H)) \otimes 
H_n(\Nat_{\Epinp}(P_F, G)) \ar[d]\\ 
  H_{n+m}(\Nat_{\Epinp}(G,I_H) \otimes \Nat_{\Epinp}(P_F, G)) \ar[d]  \\
H_{n+m} (\Nat_{\Epinp}(P_F, I_H)) = \Ext^{n+m}_{\Epinp}(F,H). }$$
Here the second map is induced by composing natural transformations. This associative pairing is called the Yoneda pairing.
\end{defn}

In particular there is a natural action of $\Ext^*_{\Epinp}(b,b) = H^*_{E_n}(b)$ on $E_n$-cohomology. One could hope to find cohomology operations via this action. For example, if the characteristic of $k$ is a prime $p$, Hochschild cohomology $HH^*(A;A_+)$ is a $p$-restricted Gerstenhaber algebra, \ie the Lie algebra structure on $\Sigma^{-1} HH^*(A;A_+)$ comes with a restriction. We will determine $H^*_{E_n}(b)$ to see whether we can find new or old cohomology operations using the Yoneda pairing.
For the remainder of this section we will denote $b\colon \Epinp \ra \kmod$ by $b_n$ since we will have to consider trees of varying levels.
Since we are going to work homologically we make $b_n^*$, the dual of $b_n$, explicit. Intuitively, $b_n^*$ is the functor assigning to a tree its set of leaves.

\begin{proposition}
The functor $b_n^*$ dual to $b_n$ assigns $k\langle[r_n]\rangle =k[\lbrace 0,...,r_n\rbrace] $ to a given tree $t=\tree$. Denoting the generators of $k\langle [r_n] \rangle$ by $\alpha_{0},..., \alpha_{r_n}$, it induces the maps 
\begin{eqnarray*}
b_n^*(\tau_n,...,\tau_{j+1}, d_i, \id,...,\id)\colon k\langle [r_n] \rangle \rightarrow k\langle [r_n] \rangle, && 
\alpha_{m}  \mapsto \alpha_{\tau_n^{-1}(m)} \\
\text{for suitable $\tau_{j+1} \in \Sigma_{[r_{j+1}]}, ...,\tau_n \in \Sigma_{[r_n]} $ as in \cite[3.5]{LR11}}, &&\\
b_n^*(d_i,\id,...,\id) \colon k\langle [r_n+1] \rangle \rightarrow k\langle [r_n] \rangle , && \alpha_{m}  \mapsto 
\begin{cases} \alpha_m, & m \leq i,\\
\alpha_{m-1}, &m >i,
\end{cases}
\\
b_n^*(\delta_i, \id, ..., \id) \colon k\langle [r_n+1] \rangle \rightarrow k\langle [r_n] \rangle , &&  \alpha_{m}  \mapsto 
\begin{cases} \alpha_m, & m < i,\\
0, & m=i,\\
\alpha_{m-1},& m >i.
\end{cases}
\end{eqnarray*}
\end{proposition}

We will show that $b^*_n$ is indeed acyclic with respect to $E_n$-homology. The case $n=1$ can be easily calculated:

\begin{proposition}\label{prop:extbHH}
For $n=1$ we have 
$$H^r_{E_1}(b_1) \cong H_r^{E_1}(b_1^*) =0$$ 
for $r>0$ and 
$$H^0_{E_1}(b_1)\cong H_0^{E_1}(b_1^*)=k.$$
\end{proposition}

For $n>1$ we derive the acyclicity of $b_n^*$ from the case $n=1$.
For this we need the following lemma.
Recall that the differential $\partial_n$ is induced by morphisms which act on the top level of a given tree. Intuitively the following lemma states that $\partial_n$ can be split into parts that correspond to morphisms acting on the different fibres.

\begin{lemma}\label{lemma:multisplithorizontal}
Let $F\colon \Epinp \ra \kmod$ be a functor and $r_1,...,r_{n-1} \geq 0$. Consider the $r_{n-1}+1$-fold multicomplex 
$$M_{x_0,...,x_{r_{n-1}}}(F) = \bigoplus_{\substack{t=\xymatrix{[x_0+...+x_{r_{n-1}}] \ar[r]^-{f_n} &... \ar[r]^-{f_2}& [r_1]   },\\ \vert f_n^{-1}(0)\vert = x_0 +1, \vert f_n^{-1}(i)\vert = x_i \text{ for all $1\leq i \leq r_{n-1}$}}}F(t)$$
with $i$-th differential $d^i$ the part of $\partial_n$ induced by morphisms operating on the fibre $f_n^{-1}(i)$.
Then 
$$\Tot(M) \cong 
\Sigma^{-r_1 - ... -r_{n-1}}(C_{(*, r_{n-1},...,r_1)}^{E_n}(F), \partial_n).$$
Furthermore we can split M into submulticomplexes corresponding to the underlying $(n-1)$-level tree $T$: Let $t^{x_0+1,x_1,...,x_{r_{n-1}}}$ be the tree extending $T$ with top level fibres of arity $x_0+1, x_1,..., x_{r_{n-1}}$. Let
$$M^T_{x_1,...,x_{r_{n-1}}} =  F(t^{x_0+1,x_1,...,x_{r_{n-1}}}).$$
Then
$$M_{*,...,*}(F) = \bigoplus_{T = \xymatrix{[r_{n-1}] \ar[r]^-{f_{n-1}} & ... \ar[r]^-{f_2} & [r_1]}} (M^T_{*,...,*}, d^0,...,d^{r_{n-1}}).$$
\end{lemma}
\begin{proof}
The differential $\partial_n$ is the sum of the maps $d^i$ for $0 \leq i \leq r_{n-1}$, each of them leaving all $1$-fibres except for $f_n^{-1}(i)$ unchanged. Two such differentials $d^i$ and $d^j$ commute except for their signs: Since $d^i$ deletes or merges edges left of $f_n^{-1}(j)$ for $i<j$, we find that $d^i d^j = - d^j d^i$. Hence it is clear that up to a shift  we can interpret $C_{(*, r_{n-1},...,r_1)}^{E_n}(F)$ as a total complex as above. 
All the differentials $d^i$ leave the lower levels of a tree $t$ as they were. Hence the splitting above holds, allowing us to consider one $(n-1)$-tree shape at a time.
\end{proof}

\begin{thm}
For all $n \geq 0$
$$H_s^{E_n}(b_n^*) = \begin{cases}
k, & s=0,\\
0,& s>0.
\end{cases}$$
\end{thm}
\begin{proof}
We will prove that $H_*(C_{(*, r_{n-1},...,r_1)}^{E_n}(b_n^*), \partial_n)$ vanishes except for $r_{n-1}=0$.
Note that if $r_{n-1}=0$ this forces $r_{n-2},..., r_1=0$, and $(C_{(*, 0,...,0)}^{E_n}(b_n^*), \partial_n)$ is isomorphic to $C_*^{E_1}(b_1^*)$. By Proposition \ref{prop:extbHH} and Lemma \ref{lemma:multisplithorizontal} this gives rise to a copy of $k$ in $H_0^{E_n}(b_n^*)$.

Now fix $r_{n-1} \geq 1, r_{n-2},...,r_1 \geq 0$.  Let $T=\xymatrix{[r_{n-1}] \ar[r]^-{f_{n-1}} & ... \ar[r]^-{f_2} & [r_1]}$ be a $(n-1)$-level tree. Consider the corresponding summand $M^T$ of the multicomplex $M(b_n^*)$ discussed in Lemma \ref{lemma:multisplithorizontal}. According to the lemma it suffices to show that the homology of the total complex associated to $M^T$ is trivial for all trees $T$ as above.
Let us start by calculating the homology of $M^T$ in the zeroth direction, \ie for each given $x_1,...,x_{r_{n-1}}\geq 1$ we consider the complex 
$$(M^T_{*,x_1,...,x_{r_{n-1}}}, d^0) =(\bigoplus_{\substack{t=\xymatrix{[*+x_1+...+x_{r_{n-1}}] \ar[r]^-{f_n} & ...\ar[r]^-{f_2} & [r_1]}, \\ \vert f_n^{-1}(0)\vert = * +1, \vert f_n^{-1}(i)\vert = x_i}}b_n^*(t), d^0).$$
Since we fixed $T$, for each $p$ there is exactly one tree $t=\tree$ with  $\vert f_n^{-1}(0)\vert = p +1$ and $ \vert f_n^{-1}(i)\vert = x_i$  for $1\leq i \leq r_{n-1}$.
Let $q=r_n-p.$ The differential $d^0$ maps $\alpha_j \in b_n^*(t) = k\langle \alpha_0,...,\alpha_{p+q}\rangle$ to 
$$
 (-1)^{n-1}b_n^*(\delta_0, \id,...,\id)(\alpha_j) 
+ \sum_{i=0}^{p-1}(-1)^{n+i} b_n^*(d_i,\id,...,\id)(\alpha_j) 
+ (-1)^{n+p}b_n^*(\delta_p,\id,...,\id)(\alpha_j).
$$
Thus for $j\leq p$ the element $d^0(\alpha_j)$ coincides up to a sign $(-1)^{n-1}$ with the image of $\alpha_j \in b_1^*([p])$ under the differential $d_{E_1}$ of $C_*^{E_1}(b_1^*)$.  If $j>p$ all the induced morphisms are the identity. Hence $(M^T_{*,x_1,...,x_{r_{n-1}}}, d^0)$ is isomorphic to
$$\xymatrix{
... \ar[r]^-{d_{E_1} \oplus 0} & \; b_1^*([3]) \oplus  k^q \; \ar[r]^-{d_{E_1} \oplus \id} & \; b_1^*([2]) \oplus  k^q\;  \ar[r]^-{d_{E_1} \oplus 0} &  \; b_1^*([1]) \oplus  k^q\;  \ar[r]^-{d_{E_1} \oplus   \id } & \; b_1^*([0]) \oplus  k^q
}$$
and $H_p(M^T_{*,x_1,...,x_{r_{n-1}}}, d^0)$ is concentrated in degree $p=0$ where it is $k$. We showed in Proposition \ref{prop:extbHH} that $H_0^{E_1}(b_1^*)=b_1^*([0])$. Hence a cycle in $H_0(M^T_{*,x_1,...,x_{r_{n-1}}}, d^0)$ is given by $\alpha_0 \in b_n^*(t^{1,x_1,...,x_{r_{n-1}}})$, where $t^{1,x_1,...,x_{r_{n-1}}}$ is the tree which extends $T$ with top level fibres of arity $1, x_1,..., x_{r_{n-1}}$.

We now determine how $d^1$ acts on these cycles. The differential $d^1$ is induced by morphisms acting on leaves in the second to left top level fibre. All of these morphisms leave the leftmost leaf invariant and therefore each of the induced maps sends $\alpha_0$ to $\alpha_0$. Hence for fixed $x_2,...,x_{r_{n-1}}\geq 1$ the chain complex $(H_0(M^T_{*,*,x_2...,x_{r_{n-1}}}, d^0), d^1)$ is one-dimensional on the generator $\alpha_0$ in each degree $r$ with differential 
$$d^1(\alpha_0) = (-1)^{2n-1}  \sum_{i=0}^{r+1}(-1)^{i} \alpha_0.$$
We see that the homology of $(H_0(M^T_{*,*,x_2...,x_{r_{n-1}}}, d^0), d^1)$ vanishes completely and  the homology of the total complex of $M^T$ is zero. Hence $(C_{(*, r_{n-1},...,r_1)}^{E_n}(b_n^*), \partial_n)$ has trivial homology as well, whenever $r_{n-1} \geq 1$.
\end{proof}

\begin{corollary}No nontrivial cohomology operations arise on $E_n$-cohomology via the Yoneda pairing defined in \ref{def:yoneda}.
\end{corollary}

\end{document}